\documentclass[12pt,a4paper,reqno]{amsart}

\usepackage{amsmath,amsfonts,amssymb,amsthm,bm,color}

\usepackage{geometry}
\usepackage{hyperref}
\hypersetup{colorlinks,linkcolor=[rgb]{0,0,1},citecolor=[rgb]{1,0,0}}

\usepackage[utf8]{inputenc}

\usepackage{multirow}

\usepackage{mathrsfs}

\usepackage{mathtools}

\usepackage{tikz-cd}

\usepackage{comment}

\usepackage{enumerate}

\newlength{\wideitemsep}
\let\olditem\item
\renewcommand{\item}{\setlength{\itemsep}{\wideitemsep}\olditem}

\renewcommand{\geq}{\geqslant}

\newcommand{\lra}{\longrightarrow}

\newcommand{\hooklongrightarrow}{\lhook\joinrel\longrightarrow}

\newtheorem{theorem}{Theorem}[section]
\newtheorem{lemma}[theorem]{Lemma}
\newtheorem{proposition}[theorem]{Proposition}
\newtheorem{corollary}[theorem]{Corollary}

\newtheorem*{conjecture*}{Conjecture}

\theoremstyle{definition}
\newtheorem{definition}[theorem]{Definition}

\newtheorem{question}[theorem]{Question}

\newtheorem{remark}[theorem]{Remark}

\newcommand{\cA}{\mathcal{A}}
\newcommand{\cB}{\mathcal{B}}
\newcommand{\cC}{\mathcal{C}}

\newcommand{\cE}{\mathcal{E}}
\newcommand{\cF}{\mathcal{F}}

\newcommand{\cI}{\mathcal{I}}

\newcommand{\cM}{\mathcal{M}}
\newcommand{\cN}{\mathcal{N}}
\newcommand{\cO}{\mathcal{O}}

\newcommand{\cS}{\mathcal{S}}
\newcommand{\cU}{\mathcal{U}}

\newcommand{\cZ}{\mathcal{Z}}

\renewcommand{\AA}{\ensuremath{\mathbb{A}}}
\newcommand{\CC}{\ensuremath{\mathbb{C}}}

\newcommand{\PP}{\ensuremath{\mathbb{P}}}

\newcommand{\ZZ}{\ensuremath{\mathbb{Z}}} 

\DeclareMathOperator{\Bl}{Bl}

\DeclareMathOperator{\Ext}{Ext}
\newcommand{\sExt}{\mathcal{E}xt}

\DeclareMathOperator{\Hilb}{Hilb}

\DeclareMathOperator{\Hom}{Hom}
\DeclareMathOperator{\Jac}{Jac}

\DeclareMathOperator{\id}{id}

\DeclareMathOperator{\Kum}{Kum}

\DeclareMathOperator{\Pic}{Pic}

\DeclareMathOperator{\Supp}{Supp}
\DeclareMathOperator{\Sym}{Sym}

\begin{document}

\begin{abstract}
    Let $M$ be a projective fine moduli space of stable sheaves on a smooth projective variety $X$ with a universal family $\cE$. We prove that in four examples, $\cE$ can be realized as a complete flat family of stable sheaves on $M$ parametrized by $X$, which identifies $X$ with a smooth connected component of some moduli space of stable sheaves on $M$.
\end{abstract}
\title{Examples of smooth components of moduli spaces of stable sheaves}

\author{Fabian Reede}
\address{Institut f\"ur Algebraische Geometrie, Leibniz Universit\"at Hannover, Welfengarten 1, 30167 Hannover, Germany}
\email{reede@math.uni-hannover.de}
\author{Ziyu Zhang}
\address{Institut f\"ur Algebraische Geometrie, Leibniz Universit\"at Hannover, Welfengarten 1, 30167 Hannover, Germany}
\curraddr{Institute of Mathematical Sciences, ShanghaiTech University, 393 Middle Huaxia Road, 201210 Shanghai, P.R.China}
\email{zhangziyu@shanghaitech.edu.cn}
\date{}
         


\maketitle

\section*{Introduction}

\subsection*{Background}

The starting point of the article is a classical result on the moduli space of stable vector bundles on curves. Let $C$ be a smooth complex projective curve of genus $g\geq 2$. We denote the moduli space of stable vector bundles on $C$ of rank $n$ with a fixed determinant line bundle $L_d$ of degree $d$ by $M$.

If $n$ and $d$ are coprime, then it is known by \cite{mum, tju} that $M$ is a fine moduli space, namely, there exist a universal vector bundle $\mathcal{E}$ on $C\times M$ with the property that the fiber $\mathcal{E}|_{C \times \{m\}}$ over a closed point $m=[E]\in M$ is isomorphic to the bundle $E$ itself. But one can also take a closed point $c\in C$ and consider the fiber
$$ \cE_c := \mathcal{E}|_{\{c\} \times M}, $$
which is a vector bundle on $M$. In \cite{nara} the authors proved that $\mathcal{E}_c$ is a simple bundle for every closed point $c\in C$ and that the infinitesimal deformation map
\begin{equation*}
    T_cC \longrightarrow \Ext^1_M(\cE_c, \cE_c)
\end{equation*}
is bijective. In fact, for all closed points $c \in C$, the bundles $\cE_c$ are stable and pairwise non-isomorphic by \cite{newstead, lange}.

Thus if we define $\mathcal{M}$ to be the moduli space of stable vector bundles on $M$ with the same Hilbert polynomial as $\mathcal{E}_c$, then the classifying morphism
\begin{equation*}
    f: C\longrightarrow \mathcal{M}, \quad c\longmapsto [\mathcal{E}_c]
\end{equation*}
identifies $C$ with a smooth connected component of $\mathcal{M}$, as explained in \cite{lange}. 

Other examples in a similar spirit appear in the pioneering work of Mukai \cite{Muk81, Muk99} on abelian varieties and K3 surfaces. In the case of K3 surfaces, Mukai considered a general polarized K3 surface $S$ of a certain degree, along with a $2$-dimensional fine moduli space $M$ of stable vector bundles of rank at least $2$ on $S$, admitting a universal family $\cE$ on $S \times M$. It turns out that $M$ is also a K3 surface, and $\cE$ can also be realized as a family of stable bundles on $M$ parametrized by $S$.

As in the previous example, we can define $\cM$ to be the moduli space of stable sheaves on $M$ with the same Hilbert polynomial as $\cE|_{\{s\} \times M}$ for any closed point $s \in S$. Mukai proved that the classifying morphism
$$ f: S \longrightarrow \cM, \quad s \longmapsto [\cE|_{\{s\} \times M}] $$
is in fact an isomorphism. In other words, $S$ can be identified with the entire moduli space of stable sheaves on $M$ with some fixed Chern classes.

\subsection*{Main result}

Motivated by the above examples, one can formulate the following question under a more general setting:

\begin{question}
	\label{ques}
	Let $X$ be a smooth projective variety and $M$ a projective fine moduli space of stable sheaves on $X$ with universal family $\mathcal{E}$ on $X\times M$. Then
	\begin{itemize}
		\item Is $\mathcal{E}$ also a flat family of stable sheaves on $M$ parametrized by $X$?
		\item If so, does the classifying map embed $X$ as a smooth connected component of some moduli space of stable sheaves on $M$?
	\end{itemize}
\end{question}

A positive answer to the above question, especially when $X$ is of low dimension and $M$ is of higher dimension, would be interesting from two perspectives. First of all, examples of stable sheaves on higher dimensional varieties (in particular on higher dimensional irreducible holomorphic symplectic manifolds) are in general difficult to construct. One important class of examples are the \emph{tautological bundles} on Hilbert schemes, which were studied in \cite{Sch10,Wan14,Wan16,Sta16}. Question \ref{ques} provides another natural approach for finding new examples. Secondly, moduli spaces of stable sheaves on higher dimensional varieties are in general badly behaved. A positive answer to Question \ref{ques} would allow us to identify some nicely behaved components of such moduli spaces, and at the same time give an explicit description of a complete family of stable sheaves over these components. 

In this article, we consider Question \ref{ques} in some of the first cases:

\begin{theorem}[Theorems \ref{thm:comp}, \ref{thm:comp2}, \ref{thm:comp3}, \ref{thm:comp4}]
	\label{thm:intro}
	Question \ref{ques} has a positive answer in the following cases:
	\begin{itemize}
    \item $X$ is a smooth projective variety of dimension $d \geqslant 2$ and $M=\Hilb^2(X)$ is the Hilbert scheme of $2$ points on $X$;
    \item $X$ is K3 surface and $M=\Hilb^n(X)$ is the Hilbert scheme of $n$ points on $X$;
    \item $X$ is an abelian surface and $M = \Kum_n(X)$ is the generalized Kummer variety of dimension $2n$ associated to $X$ for any $n \geqslant 2$;
    \item $X$ is a K3 surface of Picard rank $1$ and $M$ is some fine moduli space of stable torsion sheaves of pure dimension $1$ on $X$.
\end{itemize}
\end{theorem}

Our proof in the first of the above cases will be completely elementary. In all other cases, the moduli space $M$ is in fact an irreducible holomorphic symplectic manifold, and our proof will be divided into two steps: we first establish the flatness of $\cE$ over $X$ and the stability of the fibers $\cE_p$ over any closed point $p \in X$, then apply some very convenient results about $\PP^n$-functors (see \cite{add16}) to conclude that $X$ is in fact a component of some moduli space of stable sheaves on $M$.

It would be much more interesting to study Question \ref{ques} in more general settings, especially when $X$ and $M$ have trivial canonical classes and $\cE$ is torsion free (or even locally free) of higher rank. However, it could be then much more difficult to prove the stability of $\cE_p$ for any closed point $p \in X$. Moreover, the corresponding results about $\PP^n$-functors are not yet known to us (see \cite[Conjecture, p.2]{add16} and \cite[Conjecture 2.1]{add16-2}). 

This article consists of four sections, which are devoted to the four cases in Theorem \ref{thm:intro} respectively. The notion of $\PP^n$-functors will be briefly recalled in the beginning of \S \ref{sec2}, followed immediately by a list of $\PP^n$-functors relevant to our discussion. All schemes are defined over the field of complex numbers $\CC$.

\subsection*{Acknowledgement} We are grateful to Nicolas Addington for expert advice on the application of $\PP^n$-functors, and to Andreas Krug for communicating to us Lemma \ref{lem:KumFlat}, as well as to Benjamin Schmidt for helpful conversations. We thank the referee for carefully reading the manuscript and many helpful comments for improvements.

\section{Hilbert squares of smooth projective varieties}

Let $X$ be a smooth projective variety of dimension $d$, and $M = \Hilb^2(X)$. We denote by $\cZ \subseteq X \times M$ the universal closed subscheme and $\cI_\cZ$ the universal ideal sheaf on $X \times M$. Then we have a commutative diagram
\begin{equation}
	\label{eqn:diag1}
	\begin{tikzcd}
		\cZ \ar[rd, hook] \ar[rrd, bend left=20, "\pi"] \ar[rdd, bend right=20,  "\tau"'] & & \\
		& X \times M \ar[r] \ar[d] & M \\
		& X & 
	\end{tikzcd}
\end{equation}
where $\pi$ is a flat morphism. 

By \cite[Remark 7.2.2.]{FGA05}, we have $\cZ = \Bl_\Delta(X \times X)$, the blow-up of $X \times X$ along the diagonal $\Delta$. The projection $\tau$ can be interpreted as a composition 
\begin{equation}
	\label{eqn:diag2}
	\tau: \cZ = \Bl_\Delta(X \times X) \stackrel{b}{\longrightarrow} X \times X \stackrel{q_1}{\longrightarrow} X
\end{equation}
of the blow-up $b$ and the projection $q_1$ to the first factor. Moreover, the group $\Sigma_2=\ZZ/2\ZZ$ acts on $\cZ$ by switching the two factors, with a fixed-locus given by the exceptional divisor. By \cite[Example 7.3.1(3)]{FGA05}, $\pi$ is the quotient of $\cZ$ by $\Sigma_2$.

For any closed point $p \in X$, we write
\begin{equation*}
	F_p := \tau^{-1}(p) \subseteq \cZ \quad \text{and} \quad S_p := \pi(F_p) \subseteq M.
\end{equation*}
Then we have the following results regarding the fibers of $\tau$:

\begin{lemma}
	\label{lem:fibers}
	We have $ S_p \cong F_p \cong \Bl_p(X) $, and the morphism $\tau$ is flat.
\end{lemma}

\begin{proof}
	The morphism $\pi|_{F_p}$ can be factored into a composition
	$$ \pi|_{F_p}: F_p \hooklongrightarrow \{p\} \times M \stackrel{\cong}{\longrightarrow} M, $$
	hence $\pi$ induces an isomorphism from $F_p$ to its image $S_p$. The canonical isomorphism $F_p \cong \Bl_p(X)$ is well known. Finally, since $\cZ$ and $X$ are both smooth and the fibers $F_p$ of $\tau$ are irreducible of dimension $d$ for all closed points $p\in X$, we deduce from \cite[Theorem 23.1, Corollary]{Mat86} that $\tau$ is flat.
\end{proof}

By the description of $F_p$ as a blow-up in Lemma \ref{lem:fibers}, we denote the exceptional divisor by $E_p \stackrel{\alpha}{\hooklongrightarrow} F_p$, then $E_p \cong \PP^{d-1}$. This allows us to state the following result:

\begin{lemma}
	\label{lem:snc}
	$\pi^{-1}(S_p)$ has simple normal crossing singularities with two irreducible components
	$$ \pi^{-1}(S_p) = F_p \cup \sigma(F_p) \quad \text{such that} \quad F_p \cap \sigma(F_p) = E_p $$
	where $\sigma$ is the non-trivial element of $\Sigma_2$.
\end{lemma}

\begin{proof}
	This property can be verified analytically locally. Without loss of generality we assume that $X = \AA^n$, and $p = (0, \cdots , 0) \in X$. Then $X \times X = \AA^n \times \AA^n$ with coordinates $(x_1, \cdots, x_n, y_1, \cdots, y_n)$. We perform an affine change of coordinates: for each $1 \leqslant i \leqslant n$, we write $s_i = x_i + y_i$ and $d_i = x_i - y_i$. Then the diagonal $\Delta$ is given by
	$$ \Delta = \{ (s_1, \cdots, s_n, d_1, \cdots, d_n) \mid d_1 = \cdots = d_n = 0 \}. $$
	By \eqref{eqn:diag2} we have $\cZ = \Bl_{\Delta}(X \times X)$, which is given by a mixture of affine and projective coordinates
	$$ \Bl_\Delta(X \times X) = \{ (s_1, \cdots, s_n, d_1, \cdots, d_n, [u_1 : \cdots : u_n]) \mid [d_1 : \cdots : d_n] = [u_1 : \cdots : u_n] \}. $$
	It is covered by $n$ affine pieces, among which the first affine piece $\Bl_\Delta(X \times X)^1$ is given by $u_1 = 1$; in other words
	\begin{align*}
		\Bl_\Delta(X \times X)^1 &= \{ (s_1, \cdots, s_n, d_1, \cdots, d_n, u_2, \cdots u_n) \mid d_i = u_id_1 \text{ for } 2 \leqslant i \leqslant n \} \\
		&= \{ (s_1, \cdots, s_n, d_1, u_2, \cdots, u_n) \}.
	\end{align*}
	Then we have
	\begin{align*}
		q_1^{-1}(p) &= \{ (x_1, \cdots, x_n, y_1, \cdots, y_n) \mid x_1 = \cdots = x_n = 0 \} \\
		&= \{ (s_1, \cdots, s_n, d_1, \cdots, d_n) \mid s_i + d_i = 0 \text{ for } 1 \leqslant i \leqslant n \}.
	\end{align*}
	We write $F_p^1 = F_p \cap \Bl_\Delta(X \times X)^1$, then
	\begin{align*}
		F_p^1 &= \left\{ (s_1, \cdots, s_n, d_1, u_2, \cdots, u_n)\ \middle| \begin{array}{c} s_1 + d_1 = 0 \\ s_i + u_id_1 = 0 \text{ for } 2 \leqslant i \leqslant n \end{array} \right\} \\
		&= \left\{ (s_1, \cdots, s_n, d_1, u_2, \cdots, u_n)\ \middle| \begin{array}{c} s_1 + d_1 = 0 \\ s_i = u_is_1 \text{ for } 2 \leqslant i \leqslant n \end{array} \right\}.
	\end{align*}
	Notice that $\Bl_\Delta(X \times X)^1$ is $\sigma_2$-invariant. The action of the non-trivial element $\sigma \in \Sigma_2$ is given by
	$$ g: (s_1, \cdots, s_n, d_1, u_2, \cdots, u_n) \longmapsto (s_1, \cdots, s_n, -d_1, u_2, \cdots, u_n). $$
	Therefore we have
	$$ g(F_p^1) = \left\{ (s_1, \cdots, s_n, d_1, u_2, \cdots, u_n)\ \middle| \begin{array}{c} s_1 - d_1 = 0 \\ s_i = u_is_1 \text{ for } 2 \leqslant i \leqslant n \end{array} \right\}$$
	and the quotient $\Bl_\Delta(X \times X)^1/\Sigma_2$ is given by coordinates
	$$ \Bl_\Delta(X \times X)^1/\Sigma_2 = \{ (s_1, \cdots, s_n, e_1, u_2, \cdots, u_n) \} $$
	where $e_1 = d_1^2$. We write the image of $F_p^1$ under the quotient map by 
	$$ S_p^1 := S_p \cap \Bl_\Delta(X \times X)^1/\Sigma_2, $$
	then it follows that
	$$ S_p^1 = \left\{ (s_1, \cdots, s_n, e_1, u_2, \cdots, u_n)\ \middle| \begin{array}{c}  s_1^2=e_1 \\ s_i = u_is_1 \text{ for } 2 \leqslant i \leqslant n  \end{array} \right\}. $$
	It is now clear that
	\begin{align*}
		\pi^{-1}(S_p^1) &= \left\{ (s_1, \cdots, s_n, d_1, u_2, \cdots, u_n)\ \middle| \begin{array}{c} s_1 + d_1 = 0 \\ s_i = u_is_1 \text{ for } 2 \leqslant i \leqslant n \end{array} \right\} \\
		&\cup \left\{ (s_1, \cdots, s_n, d_1, u_2, \cdots, u_n)\ \middle| \begin{array}{c} s_1 - d_1 = 0 \\ s_i = u_is_1 \text{ for } 2 \leqslant i \leqslant n \end{array} \right\} \\
		&= F_p^1 \cup \sigma(F_p^1).
	\end{align*}
	Therefore the intersection of the two components is transverse, and given by
	$$ F_p^1 \cap \sigma(F_p^1) = \{ (s_1, \cdots, s_n, d_1, u_2, \cdots, u_n) \mid s_1 = \cdots = s_n = d_1 = 0 \} $$
	which gives precisely the exceptional divisor $E_p$ in the first affine chart, namely, $E_p \cap \Bl_\Delta(X \times X)^1$. The same argument also applies to all other affine charts of $\Bl_\Delta(X \times X)$, which finishes the proof.
\end{proof}

In the following discussion, for any closed embedding $U\hookrightarrow V$, we denote the corresponding ideal sheaf, conormal sheaf and normal sheaf by $\cI_{U/V}$, $\cC_{U/V}$ and $\cN_{U/V}$ respectively. Now we consider two smooth closed subvarieties $Y$ and $Z$ of a smooth variety, which fit in the following commutative diagram of closed embeddings:
\begin{equation}
    \label{eqn:general}
    \begin{tikzcd}
    Y\cap Z \arrow[hookrightarrow]{r}{\alpha}\arrow[hookrightarrow]{d}[swap]{i} & Z\arrow[hookrightarrow]{d}{j} \\
    Y \arrow[hookrightarrow]{r}{\delta} & Y\cup Z
    \end{tikzcd}
\end{equation}
where the intersection and the union are scheme theoretic. The following lemma will be required in our next result:

\begin{lemma}
    \label{lem:ideal}
    In the situation of \eqref{eqn:general}, we have $\cC_{Z/(Y\cup Z)}\cong \alpha_{*}\cC_{(Y\cap Z)/Y}$.
\end{lemma}

\begin{proof}
    We obtain by the second and the third isomorphism theorems that
    \begin{align*}
    	\cI_{Z/Y\cup Z} &\cong (\cI_{Y/Y\cup Z} + \cI_{Z/Y\cup Z}) / (\cI_{Y/Y\cup Z}) \\
    	&= (\cI_{Y\cap Z/Y\cup Z}) / (\cI_{Y/Y\cup Z}) \\
    	&\cong \delta_\ast \cI_{Y\cap Z/Y}.
    \end{align*} 
    Therefore we obtain
    \begin{align*}
    	\cC_{Z/(Y\cup Z)} &= j^{*}\cI_{Z/(Y\cup Z)} \\
    	&\cong j^{*}\delta_\ast\cI_{(Y\cap Z)/Y} \\
    	&\cong \alpha_\ast i^{*}\cI_{(Y\cap Z)/Y} =\alpha_\ast \cC_{(Y\cap Z)/Y}
    \end{align*}
    as required, where the second isomorphism uses \cite[\href{https://stacks.math.columbia.edu/tag/02KG}{Tag 02KG}]{stacks-project}.
\end{proof}


In our situation we pick subvarieties $Y=\sigma(F_p)$ and  $Z=F_p$ of $\cZ$ in \eqref{eqn:general}, then the morphism $\alpha$ becomes $E_p \stackrel{\alpha}{\hooklongrightarrow} F_p$. Lemma \ref{lem:ideal} immediately yields

\begin{corollary}
	\label{cor:conormal}
	We have $ \cC_{F_p/\pi^{-1}(S_p)} \cong \alpha_*\cO_{E_p}(1) $. \qed
\end{corollary}

The following result is the key to the main theorem of this section:

\begin{lemma}
	\label{lem:sections}
	If $d \geqslant 2$, then we have $\dim H^0(S_p, \cN_{S_p/M}) = d$. 
\end{lemma}

\begin{proof}
	We divide the proof in two steps.
	
	\textsc{Step 1.} We claim that $\cN_{S_p/M}$ fits into the exact sequence
	\begin{equation}
		\label{eqn:fit}
		0 \lra \cO_{F_p}^{\oplus d} \lra (\pi|_{F_p})^* \cN_{S_p/M} \lra \sExt^1_{F_p}(\alpha_*\cO_{E_p}(1), \cO_{F_p}) \lra 0.
	\end{equation}	
	We consider the chain of closed embeddings
	$$ F_p \stackrel{\iota}{\hooklongrightarrow} \pi^{-1}(S_p) \hooklongrightarrow \cZ. $$
	By \cite[Proposition 16.2.7]{EGA4}, we get the exact sequence of conormal sheaves
	\begin{equation}
		\label{eqn:seq}
		\iota^{*}\cC_{\pi^{-1}(S_p)/\cZ} \lra \cC_{F_p/\cZ} \lra \cC_{F_p/\pi^{-1}(S_p)} \lra 0.
	\end{equation}
   By Lemma \ref{lem:fibers}, $\tau: \cZ \rightarrow X$ is flat, thus by \cite[Proposition 16.2.2 (iii)]{EGA4} we get
	\begin{equation}
		\label{eqn:trivial}
		\cC_{F_p/\cZ} = (\tau|_{F_p})^*\cC_{\{p\}/X} = (\tau|_{F_p})^*\cO_{\{p\}}^{\oplus d}=\cO_{F_p}^{\oplus d}.
	\end{equation}
   Furthermore since $S_p\hookrightarrow M$ is a regular embedding of codimension $d$, the sheaf $\cC_{S_p/M}$ is locally free of rank $d$. It follows by the flatness of $\pi: \cZ \rightarrow M$ that
    \begin{equation}
		\label{eqn:nontri}
		\iota^{*}\cC_{\pi^{-1}(S_p)/\cZ} = \iota^\ast (\pi|_{\pi^{-1}(S_p)})^\ast \cC_{S_p/M} = (\pi|_{F_p})^\ast \cC_{S_p/M}
    \end{equation}
    is also locally free of rank $d$. Therefore the first two terms in \eqref{eqn:seq} are locally free sheaves of rank $d$ and the third one is by Corollary \ref{cor:conormal} torsion with support $E_p$. It follows that the first arrow in \eqref{eqn:seq} is injective. By dualizing \eqref{eqn:seq} we obtain
    \begin{equation*}
		0 \lra \cN_{F_p/\cZ} \lra (\pi|_{F_p})^* \cN_{S_p/M} \lra \sExt^1_{F_p}(\cC_{F_p/\pi^{-1}(S_p)}, \cO_{F_p}) \lra 0.
 	\end{equation*}
 	Together with \eqref{eqn:trivial}, \eqref{eqn:nontri} and Corollary \ref{cor:conormal} we obtain the claim \eqref{eqn:fit}.    

	\textsc{Step 2.} We claim that
	\begin{equation}
		\label{eqn:vanish}
		H^0(F_p,\sExt^1_{F_p}(\alpha_{*}\cO_{E_p}(1),\cO_{F_p})) = 0.
	\end{equation}	
	Indeed, the sheaf $\alpha_{*}\cO_{E_p}(1)=\alpha_{*}\cO_{E_p}(-E_p)$ admits the following resolution
   \begin{equation*}
       \begin{tikzcd}
       0 \arrow{r} & \cO_{F_p}(-2E_p) \arrow{r} & \cO_{F_p}(-E_p)\arrow{r} & \alpha_{*}\cO_{E_p}(-E_p)\arrow{r} & 0
       \end{tikzcd}
   \end{equation*}
   Dualizing this exact sequence shows 
\begin{equation*}
\sExt^1_{F_p}(\alpha_{*}\cO_{E_p}(-E_p),\cO_{F_p})=\alpha_{*}\cO_{E_p}(2E_p)=\alpha_{*}\cO_{E_p}(-2).   
\end{equation*}
   Using $E_p\cong \PP^{d-1}$ and $d \geqslant 2$, we finally get:
   \begin{equation*}
       H^0(F_p,\sExt^1_{F_p}(\alpha_{*}\cO_{E_p}(1),\cO_{F_p}))=H^0(F_p,\alpha_{*}\cO_{E_p}(-2))=H^0(E_p,\cO_{E_p}(-2))=0.
   \end{equation*}
   
   We conclude the proof by combining the long exact sequence in cohomology associated to \eqref{eqn:fit} and the vanishing result \eqref{eqn:vanish}.
\end{proof}

The following lemma is the main source for finding components of moduli spaces. The proof follows literally from \cite[Theorem 3.6]{newstead}.

\begin{lemma}
	\label{image}
	Let $X$ be a smooth projective variety of dimension $d$ and $Y$ a projective scheme. Assume that a morphism $f: X \to Y$ is injective on closed points, and $\dim T_yY = d$ for each closed point $y \in f(X)$. Then $f$ is an isomorphism from $X$ to a connected component of $Y$.
\end{lemma}

\begin{proof}
	Since $X$ is complete, $f(X)$ is a closed subvariety of $Y$ of dimension $d$. Since $\dim T_yY = d$ for each closed point $y \in f(X)$, it follows that $Y$ is smooth of dimension $d$ at each closed point $y \in f(X)$ by \cite[Theorem 6.28]{Wed10}, hence $f(X)$ must be a smooth irreducible component of $Y$, which is also a connected component of $Y$. Finally, since $f: X \to f(X)$ is a morphism between smooth projective varieties and bijective on closed points, it is an isomorphism by Zariski's Main Theorem.
\end{proof}

Combining the above results, we can now give our first main result:

\begin{theorem}
	\label{thm:comp}
	Any smooth projective variety $X$ of dimension $d \geqslant 2$ is isomorphic to a smooth connected component of a moduli space of stable sheaves with trivial determinants on $\Hilb^2(X)$, by viewing $\cI_\cZ$ as a family of coherent sheaves on $\Hilb^2(X)$ parametrized by $X$. 
\end{theorem}

\begin{proof}
   	By Lemma \ref{lem:fibers}, $\cZ$ is flat over $X$ hence $\cI_\cZ$ can be viewed as a flat family of sheaves on $\Hilb^2(X)$ parametrized by $X$. For each closed point $p \in X$, let $(\cI_\cZ)_{p}$ be the restriction of $\cI_\cZ$ on the fiber $\{ p \} \times \Hilb^2(X)$. Then $(\cI_\cZ)_{p}$ is the ideal sheaf $\cI_{S_p}$ of the closed embedding of $S_p$ into $\Hilb^2(X)$, hence is a stable sheaf of rank $1$. Therefore we obtain an induced classifying morphism
	\begin{equation}
		\label{eqn:Hilb-class}
		f: X \longrightarrow \cM, \quad p \longmapsto [ \cI_{S_p} ]
	\end{equation}
	where $\cM$ denotes the moduli space of stable sheaves on $\Hilb^2(X)$ of the class of $\cI_{S_p}$ with trivial determinants. By \cite[Lemma B.5.6]{KPS18}, $\cM$ is isomorphic to the Hilbert scheme of subschemes of $\Hilb^2(X)$ which have the same Hilbert polynomials as $S_p$ since $d \geqslant 2$. It is easy to see that $f$ is injective on closed points. Indeed, for two different closed points $p, q \in X$, $S_p$ and $S_q$ are different subschemes of $\Hilb^2(X)$ of codimension $d \geqslant 2$, hence $\cI_{S_p}$ and $\cI_{S_q}$ are non-isomorphic ideal sheaves. On the other hand, for any closed point $p \in X$, we have
	$$ T_{[\cI_{S_p}]}\cM \cong \Hom_{\Hilb^2(X)}(\cI_{S_p}, \cO_{S_p}) \cong H^0(S_p, \cN_{S_p/\Hilb^2(X)}). $$
	Hence by Lemma \ref{lem:sections}, we have
	$$ \dim T_{[\cI_{S_p}]}\cM = d. $$
	Therefore we conclude by Lemma \ref{image} that the morphism \eqref{eqn:Hilb-class} embeds $X$ as a smooth connected component of $\cM$.
\end{proof}

\section{Hilbert schemes of points on K3 surfaces}\label{sec2}

What is particular interesting to us is the case of K3 surfaces. The technique of $\PP^n$-functors allows us to obtain similar results for their Hilbert schemes of $0$-dimension subschemes of arbitrary length. We first recall the following notion of $\PP^n$-functors and its implications.

\begin{definition}{\cite[Definition 4.1]{add16}}
\label{functor}
A functor $F:\cA \rightarrow \cB$ between triangulated categories with adjoints $L$ and $R$ is called a $\PP^n$-functor if:
\begin{enumerate}[(a)]
    \item There is an autoequivalence $H$ of $\cA$ such that $$RF\cong \id\oplus H\oplus H^2\oplus\ldots\oplus H^n$$
    \item The map $$ HRF \hookrightarrow RFRF \xrightarrow{R \epsilon F} RF$$ written in components $$H\oplus H^2\oplus\ldots\oplus H^{n+1} \rightarrow \id \oplus H\oplus\ldots\oplus H^n$$ is of the form $$\begin{pmatrix} * & * & \cdots & * & * \\ 1 & * & \cdots & * & * \\ 0 & 1 & \cdots & * & * \\ \vdots & \vdots & \ddots & \vdots & \vdots \\ 0 & 0 & \cdots & 1 & * \end{pmatrix}$$
    \item We have $R\cong H^nL$. (If $\cA$ and $\cB$ have Serre functors, this is equivalent to $\cS_{\cB}FH^n\cong F\cS_{\cA}$.)
\end{enumerate}
\end{definition}

More about $\PP^n$-functors and examples can be found in \cite[\S 4]{add16}.

We will focus on the case where $\cA=D^b(X)$ and $\cB=D^b(Y)$ for two smooth projective varieties $X$ and $Y$ such that $F=\Phi_{\cF}$ is an integral functor with kernel $\cF\in D^b(X\times Y)$. In fact, we are mostly interested in the case where $\cF$ is actually a sheaf on $X\times Y$ and the autoequivalence $H=[-2]$. In this case condition (a) can be stated as
\begin{equation*}
    RF\cong \id\otimes H^\ast(\PP^n,\CC).
\end{equation*}
We will use the following simple consequence under this setting
\begin{proposition}{\cite[\S 2.1]{add16-2}}
\label{ext}
Assume $X$ and $Y$ are smooth projective varieties and $\cF$ is a coherent sheaf on $X\times Y$, flat over $X$, such that the integral functor $F=\Phi_{\cF}$ with kernel $\cF$ is a $\PP^n$-functor with associated autoequivalence $H=[-2]$. Then for any closed points $x,y\in X$ there is an isomorphism:
$$ \Ext^\ast_{Y}(\cF_x,\cF_y)\cong \Ext^\ast_{X}(\cO_x,\cO_y)\otimes H^\ast(\PP^n,\CC),$$
where $\cF_x$ and $\cF_y$ are fibers of $\cF$ over the closed points $x$ and $y$ respectively. \qed
\end{proposition}

The following list of $\PP^n$-functors will be of interest to us:
\begin{enumerate}[i)]
\item For a K3 surface $S$, $\Hilb^n(S)$ is a fine moduli space with universal ideal sheaf $\cI_{\cZ}$. The integral functor $\Phi_{\cI_{\cZ}}:D^b(S)\rightarrow D^b(\Hilb^n(S))$ is a $\PP^{n-1}$-functor with associated autoequivalence $H=[-2]$; see \cite[Theorem 3.1]{add16}.
\item Let $\Kum_n(A)$ be the generalized Kummer variety of an abelian surface $A$ with universal ideal sheaf $\cI_{\cZ}$. For any $n\geq 2$, the integral functor $\Phi_{\cI_{\cZ}}:D^b(A)\rightarrow D^b(\Kum_n(A))$ is a $\PP^{n-1}$-functor with associated autoequivalence $H=[-2]$; see \cite[Theorem 4.1]{meac}.
\item Let $S$ be a K3 surface with $\Pic(S) = \ZZ[H]$ where $H$ is an ample generator of degree $2g-2$. Assume $M$ is the fine moduli space of stable sheaves on $S$ of Mukai vector $(0, H, d+1-g)$ for some $d$ and $\cU$ is the universal sheaf over $S \times M$. Then the integral functor $\Phi_{\cU}: D^b(S) \rightarrow D^b(M)$ is a $\PP^{g-1}$-functor with associated autoequivalence $H=[-2]$; see \cite[Theorem A]{add16-2}.
\end{enumerate}

We give a first application of $\PP^n$-functors to our problem: let $S$ be a K3 surface and $M = \Hilb^n(S)$ for some positive integer $n$. Then $M$ is a fine moduli space and the ideal sheaf $\cI_{\cZ}$ of the universal family $\cZ$ is the universal sheaf on $S\times M$. It is well-known that $M$ is an irreducible holomorphic symplectic manifold. The flatness of $\cI_{\cZ}$ over $S$ follows immediately from the following result: 

\begin{lemma}{\cite[Theorem 2.1]{krug}}
	\label{lem:HilbNFlat}
   For every smooth variety $X$ and every positive integer $n$, the universal family $\cZ\subset X\times M$ is flat over $X$. \qed
\end{lemma}

The above result allows us to obtain a smooth component of the moduli space of stable sheaves on $\Hilb^n(S)$ as follows:

\begin{theorem}
	\label{thm:comp2}
	For any positive integer $n$, the K3 surface $S$ is isomorphic to a smooth connected component of a moduli space of stable sheaves on $\Hilb^n(S)$, by viewing $\cI_\cZ$ as a family of coherent sheaves on $\Hilb^n(S)$ parametrized by $S$. 
\end{theorem}

\begin{proof}
	By Lemma \ref{lem:HilbNFlat}, $\cI_\cZ$ can be viewed as a flat family of sheaves on $\Hilb^n(S)$ parametrized by $S$. For each closed point $s \in S$, let $(\cI_\cZ)_s$ be the restriction of $\cI_\cZ$ on the fiber $\{ s \} \times \Hilb^n(S)$. Then $(\cI_\cZ)_s$ is the ideal sheaf of the closed embedding of $\cZ \cap (\{ s \} \times \Hilb^n(S))$ into $\Hilb^n(S)$, hence is a stable sheaf of rank $1$. Therefore we obtain an induced classifying morphism
	\begin{equation}
		\label{eqn:HilbN-class}
		f: S \longrightarrow \cM, \quad s \longmapsto [(\cI_\cZ)_s]
	\end{equation}
	where $\cM$ denotes the moduli space of all stable sheaves on $\Hilb^n(S)$ of the class of $(\cI_\cZ)_s$. For any pair of closed points $s_0, s_1 \in S$, we obtain by \cite[Theorem 3.1]{add16} and Proposition \ref{ext} that
	\begin{equation}
		\label{eqn:HilbN-gen}
		\Ext^\ast_{\Hilb^n(S)} \big( (\cI_\cZ)_{s_0}, (\cI_\cZ)_{s_1} \big) \cong \Ext^\ast_S(\cO_{s_0}, \cO_{s_1}) \otimes H^\ast(\PP^{n-1}, \CC).
	\end{equation}
	In particular, when $s_0 \neq s_1$, it follows from \eqref{eqn:HilbN-gen} that
	$$ \Hom_{\Hilb^n(S)} \big( (\cI_\cZ)_{s_0}, (\cI_\cZ)_{s_1} \big) \cong \Hom_S(\cO_{s_0}, \cO_{s_1}) = 0, $$
	which implies that \eqref{eqn:HilbN-class} is injective on closed points; when $s_0 = s_1 = s$, it follows from \eqref{eqn:HilbN-gen} that
	$$ \Ext^1_{\Hilb^n(S)} \big( (\cI_\cZ)_s, (\cI_\cZ)_s \big) \cong \Ext^1_S(\cO_s, \cO_s), $$
	which implies that
	$$ \dim T_{[(\cI_\cZ)_s]} \cM  = \dim T_sS = 2. $$
	Therefore we conclude by Lemma \ref{image} that the morphism \eqref{eqn:HilbN-class} embeds $S$ as a smooth connected component of $\cM$, as desired.
\end{proof}

\section{Generalized Kummer varieties}

In this section we apply the technique of $\PP^n$-functors to study a component of the moduli space of stable sheaves on generalized Kummer varieties.

Let $A$ be an abelian surface and $\Hilb^{n+1}(A)$ the Hilbert scheme parametrizing closed subschemes of $A$ of length $n+1$. Let the morphism $\Sigma$ be the composition of the Hilbert-Chow morphism and the summation morphism with respect to the group law on $A$, namely
$$ \Sigma: \Hilb^{n+1}(A) \longrightarrow \Sym^{n+1}(A) \longrightarrow A, $$
then the \emph{generalized Kummer variety} is defined to be its zero fiber, namely
$$ \Kum_n(A) := \Sigma^{-1}(0), $$
which is an irreducible holomorphic symplectic manifold. If we denote the restriction of the universal subscheme over $\Hilb^{n+1}(A)$ to $\Kum_n(A)$ by $\cZ$, then we have a commutative diagram
\begin{equation*}
	\begin{tikzcd}
		\cZ \ar[rd, hook] \ar[rdd, "\varphi"', bend right=20] \ar[rrd, "\psi", bend left=15] & & \\
		& A \times \Kum_n(A) \ar[r, "p_2"'] \ar[d, "p_1"] & \Kum_n(A) \\
		& A &
	\end{tikzcd}
\end{equation*}
where $\varphi$ and $\psi$ are the compositions of the embedding and the projections. We denote the ideal sheaf of $\cZ$ in $A \times \Kum_n(A)$ by $\cI_\cZ$. It is clear that $\cI_\cZ$ is flat over $\Kum_n(A)$ since $\psi$ is flat. In fact, $\cI_\cZ$ is also flat over the other factor $A$. 

\begin{lemma}
	\label{lem:KumFlat}
	The universal ideal sheaf $\cI_\cZ$ is flat over $A$ for any $n \geqslant 2$.
\end{lemma}

\begin{proof}
	It suffices to show that the morphism $\varphi: \cZ \to A$ is flat. First of all, we claim that the dimension of the fiber $\varphi^{-1}(a_0)$ is $2n-2$ for any closed point $a_0 \in A$.
	
	On the one hand, since $A$ is smooth, the closed point $a_0 \in A$ is locally defined by two equations. Therefore locally near any point $x \in \varphi^{-1}(a_0)$, the fiber $\varphi^{-1}(a_0)$ is also defined by two equations, hence is of codimension at most $2$ by Krull's height theorem; see \cite[\S 12.I, Theorem 18]{Mat80}. In other words, we have
	\begin{equation}
		\label{eqn:geq}
		\dim\varphi^{-1}(a_0) \geqslant 2n-2.
	\end{equation}
	
	On the other hand, we have
	\begin{align*}
		\varphi^{-1}(a_0) &= \{ (a_0, \xi) \in A \times \Kum_n(A) \mid a_0 \in \Supp(\xi) \}\\
		&\cong \{ \xi \in \Kum_n(A) \mid a_0 \in \Supp(\xi) \}.
	\end{align*}
	For any such $\xi$, we can write the associated $0$-cycle $[\xi]$ as
	$$ [\xi] = \sum_{i=0}^k n_ia_i, $$
	where $a_0, a_1, \cdots, a_k$ are pairwise distinct closed points, and $n_0, n_1, \cdots, n_k$ are the multiplicities. We further require $n_1 \geqslant \cdots \geqslant n_k >0$ if $k>0$. It is clear that
	\begin{equation}
		\label{eqn:config1}
		\sum_{i=0}^k n_i = n+1
	\end{equation}
	which in particular implies $k \leqslant n$, and
	\begin{equation}
		\label{eqn:config2}
		\sum_{i=0}^k n_ia_i = 0 \in A
	\end{equation}
	which utilizes the group law on $A$. We call the partition of $n$
	$$ \vec{n} = (n_0, n_1, \cdots, n_k) $$
	the type of $\xi$. Let $\varphi^{-1}(a_0, \vec{n})$ be the set of all closed points $\xi \in \varphi^{-1}(a_0)$ of type $\vec{n}$, then we have a decomposition
	\begin{equation}
		\label{eqn:decomp_varphi}
		\varphi^{-1}(a_0)=\bigsqcup\limits_{\vec{n}} \varphi^{-1}(a_0,\vec{n}).
	\end{equation}
	We then compute the dimension of $\varphi^{-1}(a_0,\vec{n})$ for each $\vec{n}$. 
	
	When $k=0$, we have $\vec{n} = (n+1)$, and for any $\xi \in \varphi^{-1}(a_0, \vec{n})$ we have $[\xi] = (n+1)a_0$. It is clear that such $\varphi^{-1}(a_0, \vec{n})$ is non-empty if and only if $a_0 \in A$ is an $(n+1)$-torsion point. When non-empty, $\varphi^{-1}(a_0, \vec{n})$ is the punctual Hilbert scheme $\Hilb^{n+1}_{a_0}(A)$ which parametrizes length $(n+1)$ subschemes of $A$ having support at only one point $a_0$. By \cite[Corollary 1]{iar}, we have
	\begin{equation}
		\label{eqn:ausnahme}
		\dim \varphi^{-1}(a_0, \vec{n}) = n \leqslant 2n-2
	\end{equation}
	for each $(n+1)$-torsion point $a_0$ and integer $n \geqslant 2$.
	
	When $k \geqslant 1$, every $\xi \in \varphi^{-1}(a_0, \vec{n})$ corresponds to a configuration $\{ a_1, \cdots, a_k \}$ of pairwise distinct points satisfying \eqref{eqn:config2}. We can choose the first $(k-1)$ points freely, then $a_k$ is uniquely determined up to $n_k$-torsion. Hence there is a $2(k-1)$-dimensional family of configurations $\{ a_1, \cdots, a_k \}$. For any fixed configuration, the possible scheme structures on $\xi$ is classified by the product of punctual Hilbert schemes $\Hilb^{n_0}_{a_0}(A) \times \cdots \times \Hilb^{n_k}_{a_k}(A)$. By \cite[Corollary 1]{iar} and \eqref{eqn:config1}, we obtain
	\begin{align*}
		\dim \varphi^{-1}(a_0, \vec{n}) &= 2(k-1) + \sum_{i=0}^k (n_i-1) \\
		&= 2(k-1) + (n+1) - (k+1) \\
		&= n+k-2 \leqslant 2n-2.
	\end{align*}
	Combining the two cases, we have by \eqref{eqn:decomp_varphi} that
	\begin{equation}
		\label{eqn:leq}
		\dim \varphi^{-1}(a_0) \leqslant 2n-2.
	\end{equation}
	It then follows from \eqref{eqn:geq} and \eqref{eqn:leq} that all fibers $\varphi^{-1}(a_0)$ are equidimensional of dimension $2n-2$.
	
	Moreover, since $\psi$ is a surjective flat morphism and $\Kum_n(A)$ is smooth of dimension $2n$, we know $\cZ$ is Cohen-Macaulay of dimension $2n$ by \cite[Corollary 18.17]{Eisen95}. Since $A$ is smooth, we conclude that $\varphi: \cZ \to A$ is flat by \cite[Theorem 23.1, Corollary]{Mat86}, which implies that its ideal sheaf $\cI_\cZ$ is flat over $A$, as desired.
\end{proof}

\begin{remark}
	It is easy to see that the statement of Lemma \ref{lem:KumFlat} fails for $n=1$, due to the failure of \eqref{eqn:ausnahme}. In fact, in such a case, $\varphi^{-1}(a_0)$ is either a smooth rational curve or a single point, depending on whether $a_0$ is a $2$-torsion point of $A$.
\end{remark}

The above result allows us to obtain a smooth component of the moduli space of stable sheaves on $\Kum_n(A)$ as follows:

\begin{theorem}
	\label{thm:comp3}
	For any $n \geqslant 2$, the abelian surface $A$ is isomorphic to a smooth connected component of a moduli space of stable sheaves on $\Kum_n(A)$, by viewing $\cI_\cZ$ as a family of coherent sheaves on $\Kum_n(A)$ parametrized by $A$. 
\end{theorem}

\begin{proof}
	By Lemma \ref{lem:KumFlat}, $\cI_\cZ$ can be viewed as a flat family of sheaves on $\Kum_n(A)$ parametrized by $A$. For each closed point $a_0 \in A$, let $(\cI_\cZ)_{a_0}$ be the restriction of $\cI_\cZ$ on the fiber $\{ a_0 \} \times \Kum_n(A)$. Then $(\cI_\cZ)_{a_0}$ is the ideal sheaf of the closed embedding of $\cZ \cap (\{ a_0 \} \times \Kum_n(A))$ into $\Kum_n(A)$, hence is a stable sheaf of rank $1$. Therefore we obtain an induced classifying morphism
	\begin{equation}
		\label{eqn:Kum-class}
		f: A \longrightarrow \cM, \quad a_0 \longmapsto [(\cI_\cZ)_{a_0}]
	\end{equation}
	where $\cM$ denotes the moduli space of all stable sheaves on $\Kum_n(A)$ of the class of $(\cI_\cZ)_{a_0}$. For any pair of closed points $a_0, a_1 \in A$, we obtain by \cite[Theorem 4.1]{meac} and Proposition \ref{ext} that
	$$ \Ext^\ast_{\Kum_n(A)} \big( (\cI_\cZ)_{a_0}, (\cI_\cZ)_{a_1} \big) \cong \Ext^\ast_A(\cO_{a_0}, \cO_{a_1}) \otimes H^\ast(\PP^{n-1}, \CC). $$
	From here, a similar argument as in Theorem \ref{thm:comp2} shows that the morphism \eqref{eqn:Kum-class} embeds $A$ as a smooth connected component of $\cM$.
\end{proof}

\section{Moduli spaces of pure sheaves on K3 surfaces}

In this section we extend our discussion to the fine moduli spaces of stable sheaves of pure dimension $1$ on a K3 surface of Picard number $1$.

Let $S$ be a K3 surface with $\Pic(S) = \ZZ H$ where $H$ is an ample line bundle of degree $2g-2$. Let $\PP(V) \cong \PP^g$ be the complete linear system of $H$ where $V = H^0(S, H)$. Since $S$ has Picard number $1$, every curve $C$ in the linear system $\PP(V)$ is reduced and irreducible of genus $g$ with planar singularities, hence its compactified Jacobian $\overline{\Jac}^d(C)$ is reduced and irreducible of dimension $g$ by \cite[Theorem (9)]{AIK77}. We denote by $\cC$ the universal curve of the linear system $\PP(V)$. Therefore $\cC$ is a closed subscheme of $S \times \PP(V)$ and admits projections to $S$ and $\PP(V)$. All fibers of the first projection $\tau: \cC \to S$ are linear subsystems of $\PP(V)$ of codimension $1$.

Let $M$ be the moduli space of stable sheaves on $S$ with Mukai vector
$$ v = (0, H, d+1-g). $$
We assume $\gcd(2g-2, d+1-g) = 1$, then $M$ is a smooth fine moduli space of stable torsion sheaves of pure dimension $1$, hence admits a universal family $\cU$. In fact, $M$ is an irreducible holomorphic symplectic manifold. The corresponding support morphism
$$ \eta: M \longrightarrow \PP(V) $$
sends a stable sheaf to its support curve. 

Alternatively, $M$ can also be interpreted as the relative compactified Jacobian $\overline{\Jac}^d(\cC/\PP(V))$ of the family $\cC \to \PP(V)$. Hence the support of the universal family $\cU$ is given by
$$ T := \Supp(\cU) = \cC \times_{\PP(V)} M. $$
It is more convenient to consider the universal family as a sheaf on $T$, so we define
$$ \cE := \iota^\ast\cU $$
where $\iota: T \hookrightarrow S \times M$ is the closed embedding. Then we have $\cU \cong \iota_\ast\cE$ by \cite[Remark 7.35]{Wed10}.

The relation among the various spaces and morphisms introduced above can be summarised in the following commutative diagram
\begin{equation}
	\label{eqn:PureMain}
	\begin{tikzcd}
		T \ar[dd, bend right=60, "\pi"'] \ar[r, "\varphi"'] \ar[rr, bend left=20, "\psi"] \ar[d, hook, "\iota"'] & \cC \ar[r, "\tau"'] \ar[d, hook] & S \ar[d, equal] \\
		S \times M \ar[r] \ar[d] & S \times \PP(V) \ar[r] \ar[d] & S \\
		M \ar[r, "\eta"] & \PP(V) &
	\end{tikzcd}
\end{equation}
where both squares on the left are cartesian. 

Moreover, for any closed point $s \in S$, we denote the fiber $\psi^{-1}(s)$ by $T_s$, with the corresponding closed embedding $i_s: T_s \hookrightarrow T$. We also denote the pullback of $\cE$ to the fiber $T_s$ by $\cE_s$, and the pullback of $\cU$ to the fiber $\{ s \} \times M$ by $\cU_s$.

The following properties will be used later:


\begin{lemma}
	\label{lem:Tproperty}
	Both $T$ and $T_s$ (for each closed point $s \in S$) are integral and Gorenstein.
\end{lemma}

\begin{proof}
    We first note that $\cC$, being a $\PP^{g-1}$-bundle bundle over $S$, is smooth and irreducible of dimension $g+1$. Consequently $\cC$ is integral. Moreover, since both $M$ and $\PP(V)$ are smooth, and all closed fibers of $\eta$ are compactified Jacobians, which are integral of dimension $g$, the morphism $\eta$ is flat by \cite[Theorem 23.1, Corollary]{Mat86}. It follows that $\varphi$ is also flat, and every closed fiber of $\varphi$ is integral. Thus \cite[Theorem 14.44]{Wed10} implies that the generic fiber of $\varphi$ is also integral. Therefore $T$ is integral of dimension $2g+1$ by \cite[\href{https://stacks.math.columbia.edu/tag/0BCM}{Lemma 0BCM}]{stacks-project}. This means $T$ is a hypersurface in the smooth variety $S \times M$, hence $T$ is Gorenstein by \cite[Corollary 21.19]{Eisen95}.
	
	For any closed point $s \in S$, the restriction of $\varphi$ to the fibers over $s$ is given by
	$$ \varphi_s: T_s \longrightarrow \PP^{g-1}. $$
	The above properties of $\varphi$ imply that $\varphi_s$ is also flat, and that every closed fiber of $\varphi_s$ is integral. It follows for the same reason as above that $T_s$ is integral of dimension $2g-1$, hence is a hypersurface in the smooth variety $M$, which implies that $T_s$ is also Gorenstein.
\end{proof}

	

Now we turn to properties of the universal sheaf:

\begin{lemma}
	\label{lem:Eproperty}
	The sheaf $\cE$ on $T$ is flat over $S$, and the sheaf $\cE_s$ on $T_s$ is stable for each closed point $s \in S$.
\end{lemma}

\begin{proof}
	We observe that the morphism $\cC \to \PP(V)$ (the composition of the morphisms in the middle column of \eqref{eqn:Pure-class}) is projective, flat and Gorenstein of pure dimension $1$. After the base change along $\eta$, the morphism $\pi: T \to M$ (the composition of the morphisms in the left column of \eqref{eqn:Pure-class}) is also projective, flat and Gorenstein of pure dimension $1$. Furthermore $\cE$ is flat over $M$, and for any point $m \in M$, the restriction of $\cE$ to the fiber $\pi^{-1}(m)$ is torsion free. It follows by \cite[Corollary 2.2]{burb06} that
	$$ \sExt_T^i(\cE, \cO_T) = 0 $$
	for every $i>0$. Since $T$ is irreducible and Gorenstein, this implies that $\cE$ is a maximal Cohen-Macaulay sheaf on $T$.
	
	We have seen that $\varphi$ and $\tau$ are both flat morphisms, hence $\psi$ is also a flat morphism. The closed embedding $\{s\} \hookrightarrow S$ is a morphism of finite Tor dimension. After a flat base change along $\psi$, we see that $i_s: T_s \hookrightarrow T$ is also of finite Tor dimension. Since $T$ is irreducible and Gorenstein by Lemma \ref{lem:Tproperty}, \cite[Lemma 2.3 (1)]{Arinkin13} implies
	$$ L i_s^\ast \cE = i_s^\ast \cE $$
	for every closed point $s \in S$, where $Li_s^\ast$ is the derived pullback functor. It follows by \cite[Lemma 3.31]{Huy06} that $\cE$ is flat over $S$.
	
	By Lemma \ref{lem:Tproperty} we also know $T_s$ is Gorenstein, hence is in particular Cohen-Macaulay. By \cite[Lemma 2.3 (2)]{Arinkin13}, $\cE_s$ is also maximal Cohen-Macaulay, which by \cite[Satz 6.1, a) $\Rightarrow$ d)]{herz71} implies that $\cE_s$ is reflexive, and hence in particular torsion free on $T_s$. Therefore $\cE_s$ is stable since it is of rank $1$.
\end{proof}

The above result allows us to obtain again a smooth component of the moduli space of stable sheaves on $M$ as follows:

\begin{theorem}
	\label{thm:comp4}
	Under the assumptions in the present section, the K3 surface $S$ is isomorphic to a smooth connected component of a moduli space of stable sheaves on $M$, by viewing $\cU$ as a family of coherent sheaves on $M$ parametrized by $S$.
\end{theorem}

\begin{proof}
	By Lemma \ref{lem:Eproperty}, we know that the sheaf $\cU = \iota_\ast \cE$ is also flat over $S$, and the fiber $\cU_s$ is a stable sheaf on $M$ of pure dimension $2g-1$ for each closed point $s \in S$. Therefore $\cU$ is a flat family of stable sheaves on $M$ parametrized by $S$, with an induced classifying morphism given by
	\begin{equation}
		\label{eqn:Pure-class}
		f: S \longrightarrow \cM, \quad s \longmapsto [\cU_s]
	\end{equation}
	where $\cM$ is the moduli space of all stable sheaves on $M$ of the class of $\cU_s$. For any pair of closed points $s_0, s_1 \in S$, we obtain by \cite[Theorem A]{add16-2} and Proposition \ref{ext} that
	$$ \Ext^\ast_M(\cU_{s_0}, \cU_{s_1}) \cong \Ext^\ast_S(\cO_{s_0}, \cO_{s_1}) \otimes H^\ast(\PP^{g-1}, \CC). $$ 
	
	From here, a similar argument as in Theorem \ref{thm:comp2} shows that the morphism \eqref{eqn:Pure-class} embeds $S$ as a smooth component of $\cM$.
\end{proof}

\newcommand{\etalchar}[1]{$^{#1}$}

\end{document}